\documentclass[11pt]{amsart}

\usepackage{pdfsync}
\usepackage{amssymb,graphics,epsfig,overpic,amscd,subfig}
\usepackage[usenames,dvipsnames]{color}
\usepackage[colorlinks=true,linkcolor=blue,citecolor=BrickRed]{hyperref}

\newtheorem{thm}{Theorem}[section]
\newtheorem{lem}[thm]{Lemma}
\newtheorem{cor}[thm]{Corollary}
\newtheorem{prop}[thm]{Proposition}

\theoremstyle{definition}

\newtheorem{example}[thm]{Example}
\newtheorem{note}[thm]{Note}

\newcommand{\R}{\mathbf{R}}
\newcommand{\RP}{\mathbf{RP}}

\newcommand{\ol}{\overline}

\newcommand{\B}{\mathbf{B}}
\newcommand{\C}{\mathcal{C}}
\newcommand{\dir}{\operatorname{dir}}
\newcommand{\graph}{\operatorname{graph}}
\newcommand{\cone}{\operatorname{cone}}
\newcommand{\Rays}{\operatorname{Rays}}

\renewcommand{\H}{\mathbf{H}}
\renewcommand{\S}{\mathbf{S}}

\renewcommand{\tilde}{\widetilde}

\DeclareMathOperator{\dist}{dist}

\DeclareMathOperator{\rc}{rc}

\DeclareMathOperator{\nc}{nc}

\title[Tangent Cones  and Regularity]{Tangent Cones  and Regularity of \\ Real Hypersurfaces}

\author{Mohammad Ghomi}
\address{School of Mathematics, Georgia Institute of Technology,
Atlanta, GA 30332}
\email{ghomi@math.gatech.edu}
\urladdr{www.math.gatech.edu/$\sim$ghomi}

\author{Ralph Howard}
\address{Department of Mathematics,
University of South Carolina,
Columbia, SC 29208}
\email{howard@math.sc.edu}
\urladdr{www.math.sc.edu/$\sim$howard}

\date{\today \,(Last Typeset)}
\thanks{The research of the first named author was supported in part by NSF Grant DMS-0336455.}

\catcode`@=11 \@mparswitchfalse  
\newcounter{mnotecount}[section]

\begin{document}

\maketitle

\begin{abstract}
We characterize $\C^1$ embedded hypersurfaces of $\R^n$ as the only locally closed sets with continuously varying flat tangent cones whose
measure-theoretic-multiplicity is at most $m<3/2$. It follows then that any (topological) hypersurface which has flat tangent cones and is supported everywhere by balls of uniform radius is $\C^1$. In the real analytic case the same conclusion holds under the weakened hypothesis that each tangent cone  be a hypersurface. In particular, any convex real analytic hypersurface $X\subset\R^n$ is $\C^1$. Furthermore, if $X$ is  real algebraic, strictly convex, and unbounded, then its projective closure is a $\C^1$ hypersurface as well, which shows that $X$ is the graph of a function defined over an entire hyperplane. Finally we show that the last property is a special feature of real algebraic sets, in the sense that it does not hold in the real analytic category.
 \end{abstract}

\section{Introduction}
The \emph{tangent cone} $T_p X$ of a set $X\subset\R^n$ at a  point $p\in X$ consists of the limits of all (secant) rays which originate from $p$ and pass through a sequence of points $p_i\in X\smallsetminus\{p\}$ which converges to $p$. These objects, which are generalizations of tangent spaces to smooth submanifolds,  were first used by Whitney \cite{whitney:tangentsII,whitney:tangents} to study the singularities of real analytic varieties, and also play a fundamental role in  geometric measure theory \cite{federer:book, morgan:book}. 
Here we employ tangent cones to study the differential regularity  of (topological) \emph{hypersurfaces}, i.e., subsets of $\R^n$ which are locally homeomorphic to $\R^{n-1}$, and apply these results in the intersection of real algebraic geometry with convexity theory. In particular we show that there are  unexpected geometric differences between real algebraic and real analytic convex sets. 
Our  initial result  is the following  characterization for  $\C^1$ hypersurfaces, i.e., hypersurfaces which may be represented locally as the graph of continuously differentiable functions. Here a \emph{locally closed} set is the intersection of a closed set with an open set; further,  $T_p X$ is \emph{flat} when it is a hyperplane, and
the  (lower measure-theoretic) \emph{multiplicity} of $T_pX$ is  then just the $(n-1)$-dimensional lower density of $X$ at $p$ (see Section \ref{subsec:multiplicity}). 

\begin{thm}\label{thm:main1}
Let $X\subset\R^n$  be a locally closed set. 
Suppose that $T_pX$ is flat for each $p\in X$, and  depends continuously on $p$.
Then $X$ is a union  of $\C^1$ hypersurfaces. Further, if the multiplicity of each $T_pX$ is at most $m<3/2$, then $X$ is a  hypersurface.
\end{thm}

See \cite{bg} for other  recent results on  characterizing  $\C^1$ submanifolds of $\R^n$ in terms of their  tangent cones, and extensive references in this area, including  works of Gluck \cite{gluck:tan1, gluck:tan2}. Next we develop some applications of Theorem \ref{thm:main1} by looking for natural settings where
the continuity of $p\mapsto T_pX$ follows automatically as soon as $T_p X$ is flat. This is the case for instance when $X$ is a  \emph{convex hypersurface}, i.e., the boundary of a convex set with interior points. 
More generally, the flatness of $T_pX$ imply their continuity whenever $X$ has \emph{positive support} (Lemma \ref{lem:posreach1}), which means that through each point of $X$ there passes a ball $B$ of  radius $r$, for some uniform constant $r>0$, such that the interior of $B$ is disjoint from $X$. If there pass two such balls with disjoint interiors through each point of $X$, then we  say $X$ has \emph{double positive support}. The above theorem together with an observation of H\"{o}rmander (Lemma \ref{lem:hormander}) yield:

\begin{thm}\label{thm:main2}
Let $X\subset\R^n$ be a  locally closed set with flat tangent cones and positive support. Suppose that either $X$ is a hypersurface, or the multiplicity of each $T
_pX$ is at most $m<3/2$. Then $X$ is a $\C^1$ hypersurface; furthermore,  if $X$ has \emph{double} positive support,  then it is $\C^{1,1}$.
\end{thm}

Here $\C^{1,1}$ means that $X$ may be represented locally as the graph of a function with \emph{Lipschitz} continuous derivative \cite[p. 97]{hormander}.
Examples of sets with positive support include the boundary of sets of \emph{positive reach} introduced by Federer \cite{federer:curvature,thale}:
A closed set $X\subset\R^n$ has positive reach if there exists a neighborhood  of it in $\R^n$ of radius $r>0$ such that  every point in that neighborhood is closest to a \emph{unique} point in $X$. For instance all convex sets have positive reach. Thus the  above theorem  generalizes the well-known fact that convex hypersurfaces with unique support hyperplanes; or, equivalently, convex functions with unique subdifferentials, are $\C^1$ (Lemma \ref{lem:convexity}); later we will offer another generalization as well (Theorem \ref{thm:posreach}). Also note  that any hypersurface with positive reach  automatically has double positive support, which we may think of  as having a ball ``roll freely" on both sides of $X$. This is important for $\C^{1,1}$ regularity 
as there are hypersurfaces with flat tangent cones and positive support which are not $\C^{1,1}$ (Example \ref{ex:deltoids}).
Finally we should mention that the last claim in Theorem \ref{thm:main2} is closely related to a general result of Lytchak \cite{lytchak}, see Note \ref{note:lytchak}.
Next we show that there are settings where the assumption in the above results that $T_pX$ be flat may be weakened. A set $X\subset\R^n$ is \emph{real analytic} if locally it coincides with the zero set of an analytic function. Using a version of the real nullstellensatz, we will show (Corollary \ref{cor:thmprop}) that if any tangent cone $T_pX$ of a real analytic hypersurface is  itself a hypersurface, then it is symmetric with respect to $p$.  Consequently Theorem \ref{thm:main2} yields: 

\begin{thm}\label{thm:main3}
Let $X\subset\R^n$ be a real analytic  hypersurface with positive support. If $T_p X$ is a hypersurface for all $p$ in $X$, then $X$ is $\C^1$. In particular, convex real analytic hypersurfaces are $\C^1$.
\end{thm}

Note that by a hypersurface in this paper we  always mean a set which is locally homeomorphic to $\R^{n-1}$.
The last result is optimal in several respects. First, the assumption that each $T_p X$ be a hypersurface is necessary due to the existence of cusp  type singularities which arise for instance in the planar curve $y^2=x^3$. Further, there are real analytic hypersurfaces with flat tangent cones which are supported at each point by a ball of (nonuniform) positive radius, but are not $\C^1$ (Example \ref{ex:analytic}). In particular, the assumption on the existence of positive support is necessary; although the case $n=2$ is an exception (Section \ref{subsec:cusp}). Furthermore, there are real analytic convex hypersurfaces which are not $\C^{1,1}$, or even $\C^{1,\alpha}$, and thus the $\C^1$ conclusion in the above theorem is  sharp (Example \ref{ex:deltoids}). Finally there are convex \emph{semi}-analytic hypersurfaces which are not $\C^1$ (Example \ref{ex:dingdong}).
Next  we  show that the  conclusions in the above theorem may be strengthened when $X$ is \emph{real algebraic}, i.e., it is the zero set of a polynomial function. We say that  $X\subset\R^n$ is an \emph{entire graph} if there exists a line $\ell\subset\R^n$ such that $X$ intersects every line parallel to $\ell$ at exactly one point. In other words, after a rigid motion, $X$ coincides with the graph of a function $f\colon\R^{n-1}\to\R$. The \emph{projective closure} of $X$ is its closure as a subset of $\RP^n$ via the \emph{standard embedding} of $\R^n$ into $\RP^n$ (see Section \ref{subsec:projtranform}). Also recall  that a convex hypersurface is \emph{strictly convex} provided that it does not contain any line segments.

\begin{thm}\label{thm:main4}
Let $X\subset\R^n$ be a real algebraic convex hypersurface homeomorphic to $\R^{n-1}$. Then $X$ is an entire graph. Furthermore, if $X$ is strictly convex, then its projective closure is a $\C^1$ hypersurface.
\end{thm}

Thus, a real algebraic strictly convex hypersurface is not only $\C^1$ in $\R^n$ but also at  ``infinity".
Interestingly, this phenomenon does not in general hold when $X$ is the zero set of an analytic function (Example \ref{ex:horseshoe}). 
So Theorem \ref{thm:main4} demonstrates that there are genuine geometric differences between the categories of real analytic and real algebraic convex hypersurfaces, as we had mentioned earlier. 
Further, the above theorem does not hold when $X$ is \emph{semi}-algebraic  (Example \ref{ex:dingdong}),
and the  convexity assumption here is necessary, as far as the entire-graph-property is concerned (Example \ref{ex:teardrop}). Finally, the regularity of the projective closure does not hold without the strictness of convexity (Example \ref{ex:pcl}).  
See \cite{ranestad&sturmfels,sss,sturmfels&uhler} for some examples of convex hypersurfaces arising in algebraic geometry.
\emph{Negatively curved} real algebraic hypersurfaces have also been studied  in \cite{connell&ghomi}.
Finally, some recent  results on geometry and regularity of locally convex hypersurfaces may  be found in \cite{alexander&ghomi:chp,alexander&ghomi:chpII, agw,ghomi:stconvex}.

\section{Preliminaries: Basics of Tangent Cones and Their \\
Measure Theoretic Multiplicity}\label{sec:prelim}
\subsection{}\label{subsec:prelim1}
Throughout this paper $\R^n$ denotes the $n$-dimensional Euclidean space with origin $o$, standard inner product $\langle \,,\,\rangle$ and norm $\|\cdot\|:=\langle\cdot,\cdot\rangle^{1/2}$. The unit ball $\B^n$ and sphere $\S^{n-1}$ consist respectively of points $x$ in $\R^n$ with $\|x\|\leq 1$ and $\|x\|=1$. For any sets $X$, $Y\subset\R^n$, $X+Y$ denotes their Minkowski sum or the collection of all $x+y$ where $x\in X$ and $y\in Y$. Further for any constant $\lambda\in\R$, $\lambda X$ is the set of all $\lambda x$ where $x\in X$, and we set $-X:=(-1)X$.
For any point $p\in \R$, and $r>0$, $B^n(p,r):=p+r\B^n$ denotes the (closed) \emph{ball} of radius $r$ centered at $p$. 
By a \emph{ray} $\ell$ with origin $p$ in $\R^n$ we mean the set of points given by $p+tu$, where $u\in\S^{n-1}$ and $t\geq 0$. We call $u$ the \emph{direction} of $\ell$, or $\dir(\ell)$. The set of all rays in $\R^n$ originating from $p$ is denoted by $\Rays(p)$. This space may be topologized by defining the distance between a pair of rays $\ell$, $\ell'\in \Rays(p)$ as 
$$
\dist(\ell,\ell'):=\|\dir(\ell)-\dir(\ell')\|.
$$ 
Now for any set $X\subset\R^n$, and $p\in X$, we may define the \emph{tangent cone} $T_p X$ as the union of all rays $\ell\in \Rays(p)$ which meet the following criterion: there exists a sequence of points $p_i\in X\smallsetminus\{p\}$ converging to $p$ such that the corresponding sequence of rays $\ell_i\in\Rays(p)$ which pass through $p_i$  converges to $\ell$. In particular note that $T_p X=\emptyset$ if and only if $p$ is an isolated point of $X$.

We should point out that our definition of tangent cones coincides with  Whitney's \cite[p. 211] {whitney:tangentsII}, and is consistent with the usage of the term in geometric measure theory  \cite[p. 233]{federer:book}; however, in algebraic geometry,  the tangent cones are usually defined as the limits of secant \emph{lines} as opposed to \emph{rays} \cite{shafarevich,OW}. We refer to the latter notion as the  \emph{symmetric} tangent cone. More explicitly, let $(T_p X)^*$ denote the reflection of $T_p X$ with respect to $p$, i.e., the set of  all rays $\ell\in \Rays(p)$ with $\dir(\ell)=-u$ where $u$ is the direction of a ray in $T_p X$; then  the \emph{symmetric tangent cone} of $X$ at $p$ is defined as:
$$
\tilde T_p X:=T_pX\cup (T_p X)^*.
$$
 Still a third usage of  the term tangent cone in the literature,  which should not be confused with the two notions we have already mentioned, corresponds to the zero set of  the leading terms of the Taylor series of analytic functions $f$ when $X=f^{-1}(0)$, and may be called the \emph{algebraic tangent cone}. The relation between these three notions of tangent cones, which  will be studied in Section \ref{sec:tangents}, is the key to proving Theorem \ref{thm:main3}.

Next we will record a pair of basic properties of tangent cones.
For any $\ell\in \Rays(p)$ and $\delta>0$, let $C(\ell,\delta)$ be a neighborhood of $\ell$ of radius $\delta$ in $\Rays(p)$, i.e., 
the cone given by
$$
C(\ell,\delta):=\big\{\,\ell'\in \Rays(p)\mid \|\dir(\ell)-\dir(\ell')\|\leq\delta\,\big\}
$$ 
 The following observation is a quick consequence of the definitions which will be used frequently throughout this work:

\begin{lem}\label{lem:cone}
Let $X\subset\R^n$, $p\in X$, and $\ell\in \Rays(p)$.  Then $\ell\subset T_pX$, if, and only if, for every $\delta$, $r>0$,  
\begin{equation*}\big(X\smallsetminus\{p\}\big)\cap C(\ell,\delta)\cap B^n(p,r)\neq\emptyset.\end{equation*}
In particular, if every open neighborhood of $p$ in $X$ intersects $\ell\smallsetminus\{p\}$, then $\ell\subset T_p X$. Furthermore, $T_p X$ is always a closed set.
\qed
\end{lem}

Another useful way to think of $T_pX$ is as a certain limit of homotheties of $X$, which we describe next.
For any $\lambda>0$, let  
$$
X_{p,\lambda}:=\lambda (X-p)+p
$$
be the \emph{homothetic expansion} of $X$ by the factor $\lambda$ centered at $p$.
Then we have
\begin{lem}\label{lem:coneconverge}
Let $X\subset \R^n$, $p\in X$, and $r>0$. Then for any $\epsilon>0$ there exists $\lambda>0$ such that for every $\lambda'\geq\lambda$,
\begin{equation}\label{eq:coneconverg1}
X_{p,\lambda'}\cap B^n(p,r) \subset (T_pX+\epsilon \B^n)\cap B^n(p,r).
\end{equation}
\end{lem}
\begin{proof}
We may assume, for convenience, that $p=o$ and  $r=1$. Set $B(T_oX,\epsilon):=T_oX+\epsilon \B^n$, and $X_\lambda:=X_{o,\lambda}$. Let $\epsilon<1$ be given and suppose, towards a contradiction, that 
there is a sequence $\lambda_i\to\infty$ such that for each $\lambda_i$, there is a point
$\lambda_i x_i\in X_{\lambda_i}\cap\B^n$ which does not belong to $B(T_oX,\epsilon)$. Since $\B^n$ is compact, $\lambda_i x_i$ has a limit point $\ol x$ in $\B^n$, which will be disjoint from the interior of $B(T_oX,\epsilon)$. Thus the ray $\ol\ell\in\Rays(o)$ which passes through $\ol x$ does not belong to $T_o X$. On the other hand $\ol\ell$ is a limit of rays $\ell_i$ which originate from $o$ and pass through $x_i$. Further note that since $\|\lambda_ix_i\|\leq 1$ and $\lambda_i\to\infty$, we have $x_i\to o$. Thus $\ol\ell$ must belong to $T_o X$ and we have a contradiction.
\end{proof}

It is now natural to wonder whether $\lambda$ in Lemma \ref{lem:coneconverge} may be found so that the following inclusion, also holds for all $\lambda'\geq\lambda$:
\begin{equation}\label{eq:coneconverg2}
(X_{p,\lambda'}+\epsilon \B^n)\cap B^n(p,r) \subset T_pX\cap B^n(p,r).
\end{equation}
For if \eqref{eq:coneconverg1} and \eqref{eq:coneconverg2} were both true, then it would mean that $X_{p,\lambda'}$ converges to $T_p X$ with respect to the \emph{Hausdorff distance} \cite{rw} inside any ball centered at $p$. In fact when $X$ is real analytic, this is precisely the case \cite[p. 220]{whitney:tangentsII}. This phenomenon, however, does not in general hold: consider for instance the set $X\subset\R$ which consists of the origin $o$ and  points
$\pm 1/2^n$, 
$n=1,2,3,\dots$ (then $T_oX=\R$, while $X_{o,2^n}\cap\B^1=X$).
Still, there does exist a natural notion of convergence with respect to which $T_p X$ is the limit of the homothetic expansions $X_{p,\lambda}$: A set $X\subset\R^n$ is said to be the \emph{outer limit} \cite{rw} of a sequence of sets $X_i\subset\R^n$, and we write $\limsup_{i\to\infty} X_i=X$, provided that for every $x\in X$ there exists a subsequence $X_{i_j}$ which eventually intersects every open neighborhood of $x$. 
\begin{cor}\label{cor:outer}
For any set $X\subset\R^n$ and point $p\in X$, the tangent cone $T_p X$ is the outer limit of the homothetic expansions of $X$ centered at $p$:
$$T_p X=\limsup_{\lambda\to\infty} X_{p,\lambda}.$$
\end{cor} 
\begin{proof}
Lemma \ref{lem:coneconverge} quickly yields that $\limsup_{\lambda\to\infty} X_{p,\lambda}\subset T_p X$, and the reverse inclusion follows from Lemma \ref{lem:cone}, due to the invariance of $T_pX$ under homotheties.
\end{proof}

\subsection{}\label{subsec:multiplicity} Now we describe how one may assign a \emph{multiplicity} to each tangent cone  of a set $X\subset\R^n$. Let $\mu$ be a measure on $\R^n$. Then the (lower) multiplicity of $T_p X$ with respect to $\mu$ will be defined as
$$
m_\mu(T_p X):=\liminf_{\lambda\to\infty} \frac{\mu\Big(X_{p,\lambda}\cap B^n(p,r)\Big)}{\mu\Big(T_p X\cap B^n(p,r)\Big)},
$$
for some $r>0$. Since $T_p X$ is invariant under homotheties of $\R^n$ centered at $p$, it follows  that $m_\mu(T_p X)$ does not depend on $r$. If the Hausdorff dimension of $T_p X$ is an integer $d$, then we define the multiplicity of $T_p X$ with respect to the \emph{Hausdorff measure} $\mathcal{H}^d$ as
$$
m(T_p X):=m_{\mathcal{H}^d}(T_p X).
$$
One may also note that $m(T_pX)$ is simply the ratio of the $d$-dimensional \emph{lower densities}  $\mathcal{D}^d$ of $X$ and $T_p X$ at $p$. More explicitly, let us recall that for any set $X\subset\R^n$ at a  point $p$ 
$$
\mathcal{D}^d(X,p):=\liminf_{r\to 0}\frac{\mathcal{H}^d\big(X\cap B^n(p,r)\big)}{\mathcal{H}^d\big(r\B^d\big)}.
$$
 Now it is not difficult to check that
 $$
 m(T_p X)=\frac{\mathcal{D}^d(X,p)}{\mathcal{D}^d(T_p X,p)}.
 $$
 Note in particular that when $T_p X$ is an affine $d$-dimensional space, then $\mathcal{D}^d(T_p X,p)=1$ and therefore 
 $m(T_p X)=\mathcal{D}^d(X,p)$.
Figure \ref{fig:y} 
\begin{figure}[h]
   \centering
    \begin{overpic}[height=0.8in]{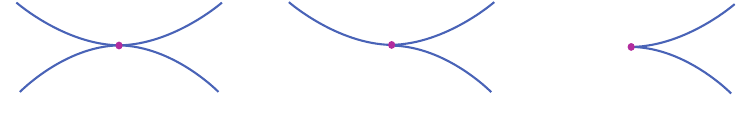} 
    \put(11.5,3){$m=2$}
     \put(46,3){$m=3/2$}
      \put(80,3){$m=2$}
      \put(11.7,-1){$\mathcal{D}=2$}
     \put(46.3,-1){$\mathcal{D}=3/2$}
      \put(80.3,-1){$\mathcal{D}=1$}
    \end{overpic}
   \caption{}
   \label{fig:y}
\end{figure}
illustrates some examples of multiplicity and density of sets $X\subset\R^2$ at the indicated points.

\section{Sets with Continuously Varying Flat  Tangent Cones: \\Proof of Theorem \ref{thm:main1}}
The main idea for proving Theorem \ref{thm:main1} is to show that $X$ can be represented locally as a multivalued graph with regular sheets. To this end, we first  need to record a pair of basic facts concerning graphs with flat tangent cones. The first observation is an elementary characterization of the graphs of $\C^1$ functions in terms of the continuity and flatness of their tangent cones:

\begin{lem}\label{lem:C1graph}
Let $U\subset\R^{n-1}$ be an open set,  $f\colon U\to\R$ be a function, and set $\ol a:=(a,f(a))$ for all $a\in U$. Then $f$ is $\C^1$ on $U$, if $T_{\ol a}\graph(f)$ is a hyperplane which depends continuously on $a$ and is never orthogonal to $\R^{n-1}$.
\end{lem}
\begin{proof}
First suppose that $n=2$, and for any $a\in U$ let $m(a)$ be the slope of $T_{\ol a}\graph(f)$. Then for any sequence of points $x_i\in U$ converging to $a$,  the slope of the (secant) lines through $\ol a$ and $\ol x_i$ must converge to $m(a)$. So
$
f'(a)=m(a),
$
which shows that $f$ is differentiable at $a$. Further, since $T_{\ol a}\graph(f)$ depends continuously on $a$, $m$ is continuous. So $f$ is $\C^1$. 
Now for the general case where $U\subset\R^{n-1}$, let $e_i$, $i=1,\dots n-1$,  be the standard basis for $\R^{n-1}$. Then applying the same argument to the functions $f_i(t):=f(a+(t-a)e_i),$  shows that 
$\partial{f}/{\partial x_i}(a)=f_i'(a)=m_i(a)$ where $m_i(a)$ is the slope of the line in $T_{\ol a}\graph(f)$ which projects onto  $a+te_i$. Thus all partial derivatives $\partial{f}/{\partial x_i}$ exist and are continuous on $U$; therefore, it follows that $f$ is $\C^1$.
\end{proof}

Our next observation yields a simple criterion for checking that the tangent cone of a graph is flat:

\begin{lem}\label{lem:Tgraph}
Let $U\subset\R^{n-1}$ be an open set,  and $f\colon U\to\R$ be a continuous function. Suppose that, for some $a\in U$,  $T_{\ol a}\graph(f)$ lies in a hyperplane $H$ which is not  orthogonal to $\R^{n-1}$. Then $T_{\ol a}\graph(f)= H$.
\end{lem}
\begin{proof}
Suppose, towards a contradiction, that there exists a point $x\in H$, which does not belong to $T_{\ol a}\graph(f)$, and let $\ell$ be the ray which originates from ${\ol a}$ and passes through $x$ ($\ell$ is well-defined since $x$ must necessarily be different from ${\ol a}$). Then, since $T_{\ol a}\graph(f)$ is invariant under homothetic expansions  of $\R^n$ centered at $\ol a$, the open ray
 $\ell \smallsetminus\{{\ol a}\}$ must be disjoint from $T_{\ol a} \graph(f)$. So, by Lemma \ref{lem:cone}, there exists $\delta$, $r>0$ such that $C(\ell,\delta)\cap B^n({\ol a},r)$ is disjoint from $\graph(f)\smallsetminus\{{\ol a}\}$. Now let $\pi\colon\R^n\to\R^{n-1}$ be projection into the first $n-1$ coordinates. Then $\pi(\ell)$ is a ray in $\R^{n-1}$, since $H$ is not orthogonal to $\R^{n-1}$. Let $x_i\in \pi(\ell)$ be a sequence of points converging to $\pi({\ol a})$, then $\ol x_i:=(x_i,f(x_i))$ converges to ${\ol a}$ by continuity of $f$. Consequently any ray $\ell'$ which is a limit of  rays $\ell_i$ which pass through $\ol x_i $ and $\ol a$ must belong to $T_{\ol a}\graph(f)$.  Then $\ell'\neq \ell$ and  $\ell'\subset H$.  So $\pi(\ell')\neq\pi(\ell)$, since $H$ is not orthogonal to $\R^{n-1}$. But, since $\pi(\ol x_i)=x_i$, we must have $\pi(\ell_i)=\pi(\ell)$, which yields $\pi(\ell')=\pi(\ell)$, and we have a contradiction.
\end{proof}

Next recall that a set $X\subset\R^n$ is \emph{locally closed} if it is the intersection of an open set with a closed set. An equivalent formulation is that for every point $p$ of $X$ there exists $r>0$ such that $V:=B^n(p,r)\cap X$ is closed in $X$. Note that then the boundary $\partial V$ of $V$ as a subset of $X$ lies on  $\partial B^n(p,r)$  and therefore is disjoint from $p$.

\begin{lem}\label{lem:projection}
Let $X\subset\R^n$ be a locally closed set. Suppose that for all $p\in X$, $T_p X$ is flat and  depends continuously on $p$. 
Let $H$ be a hyperplane which is not orthogonal to $T_p X$, for some $p\in X$, and $\pi\colon\R^n\to H$ be the orthogonal projection. Then there exists an open neighborhood $U$ of $p$ in $X$ such that $\pi|_U$ is an open mapping.
\end{lem}
\begin{proof}
By continuity of $p\mapsto T_p X$, we may choose an open neighborhood $U$ of $p$  so small that no tangent cone  of $U$ is orthogonal to $H$. We claim that $\pi$ is then an open mapping on $U$. To see this, let $q\in U$, and set
$$
V:=B^n(q,r)\cap X.
$$ 
Choosing $r>0$ sufficiently small, we may assume that $V\subset U$. Further, since $X$ is locally closed, we may assume that $V$ is compact, and it also has the following two properties:  first,  the boundary $\partial V$ of $V$ in $X$ lies on $\partial B^n(q,r)$, 
which implies that 
$$
q\not\in\partial V;
$$
second, by Lemma \ref{lem:cone},  $V$ intersects  the line $L_q:=\pi^{-1}(\pi(q))$ which passes through $q$ and is orthogonal to $H$ only at $q$, for otherwise $L_q$ would have to belong to $T_q X$, which would imply that $T_q X$ is orthogonal to $H$.
So we conclude that
$$
\pi(q)\not\in \pi(\partial V).
$$
 But recall that $V$ is compact, and so $\partial V$ is compact as well. Thus there exists, for some $s>0$,  a closed ball $B(\pi(q),s)\subset H$  such that
\begin{equation}\label{eq:pidv}
B\big(\pi(q),s\big)\cap \pi(\partial V)=\emptyset.
\end{equation}
Now suppose towards a contradiction that $\pi(q)$ is not an interior point of $\pi(V)$. Then $B(\pi(q),\delta)\not\subset \pi(V)$, for any $\delta>0$. In particular there is a point
$$
x\in B\big(\pi(q),s/2\big)\smallsetminus\pi(V).
$$ 
Since $x\not\in\pi(V)$, and $\pi(V)$ is compact, there exists  a ball  $B(x,t)\subset H$ which intersects $\pi(V)$ while the interior of $B(x,t)$ is disjoint from $\pi(V)$. Note that, since $q\in V$, we have 
$t\leq\|x-\pi(q)\|\leq s/2$, and so for any $y\in B(x,t)$,
$$
\|y-\pi(q)\|\leq \|y-x\|+\|x-\pi(q)\|\leq t+s/2\leq s,
$$
which  yields
$$
B(x,t)\subset B(\pi(q),s).
$$
Thus $B(x,t)\cap\pi(\partial V)=\emptyset$ by \eqref{eq:pidv}.
So  the cylinder $C:=\pi^{-1}(B(x,t))$ supports $V$ at some point $y$ in the interior of $V$. Then the tangent cone $T_yX=T_y V$ must be tangent  to $C$, and hence be orthogonal to $H$, which is a contradiction, since $V\subset U$, and $U$ was assumed to have no tangent cones orthogonal to $H$.
\end{proof}

Now we are ready to prove the main result of this section:

\begin{proof}[Proof of Theorem  \ref{thm:main1}]
We will divide the argument into two parts. First, we show that $X$ is a union of $\C^1$ hypersurfaces, and then we show that it is a hypersurface with the assumption on multiplicity.

\emph{Part I.} Assume that $o\in X$ and $T_oX$ coincides with the hyperplane $x_n=0$, which we refer to as $\R^{n-1}$. Let  $\pi\colon\R^n\to\R^{n-1}$ be the projection into the first $n-1$ coordinates. Then, since the $x_n$-axis does not belong to $T_o X$, it follows from Lemma \ref{lem:cone} that there exists an open neighborhood $U$ of $o$ in $X$ which intersects the $x_n$-axis only at $o$.
Further, by Lemma \ref{lem:projection}, we may choose $U$ so small that  $\pi(U)$ is open. So $\Omega:=B^{n-1}(o,r)\subset\pi(U)$ for some $r>0$. Now let $U':=\pi^{-1}(\Omega)\cap U$.  Since $X$ is locally closed, we may assume that $U'$ is compact. Further, since $\pi(U')=\Omega$, we may define a function $f\colon \Omega\to\R$ by letting  $f(x)$ be the supremum of the height (i.e., the $n^{th}$-coordinate) of $\pi^{-1}(x)\cap U'$ for any $x\in \Omega$. Then $f(o)=0$, so $o\in\graph(f)$. Further we claim that $f$ is continuous. Then Lemma \ref{lem:Tgraph} yields that the tangent cones of $\graph(f)$ are hyperplanes. Of course the tangent hyperplanes of $\graph(f)$ must also vary continuously  since $\graph(f)\subset X$. So  $f$ must be $\C^1$ by Lemma \ref{lem:C1graph}. Consequently, the graph of $f$ is a $\C^1$ embedded disk in $X$ containing $o$. This shows that each point $p$ of $X$ lies in a $\C^1$ hypersurface, since after a rigid motion we can always assume that $p=o$ and $T_p X=\R^{n-1}$.

So to complete the first half of the proof, i.e., to show that $X$ is composed of $\C^1$ hypersurfaces, it remains only to verify that the function $f$ defined above is continuous. To see this, first note that for every $x\in \pi(U')$, the set $\pi^{-1}(x)\cap U'$ is discreet, since none of the the tangent hyperplanes of $U$ contain the line $\pi^{-1}(x)$, and thus we may apply Lemma \ref{lem:cone}. 
In particular, since $U'$ is bounded,  it follows that $\pi^{-1}(x)\cap U'$ is finite. So we may let $\ol x$ be the ``highest" point of $\pi^{-1}(x)\cap U'$, i.e., the point in $\pi^{-1}(x)$ with the largest $n^{th}$ coordinate.  Now to establish the continuity of $f$, we just have to check that if $x_i\in \Omega$ form a sequence of points converging to $x\in \Omega$, then the corresponding points $\ol x_i$ converge to $\ol x$. To see this, first recall that  $\ol x_i$  must have a limit  point $\ol x'\in U'$,  since $U'$ is compact. So $\ol x'$ cannot lie above $\ol x$, by definition of $\ol x$. Suppose, towards a contradiction, that $\ol x'$ lies strictly below $\ol x$. By Lemma \ref{lem:projection}  there are then open neighborhoods $V$ and $V'$ of $\ol x$ and $\ol x'$ in $U$ respectively such that $\pi(V)$ and $\pi(V')$ are open in $\R^{n-1}$. Further, we may suppose that $V'$ lies strictly below $V$. Note that the sequence $\ol x_i$ must eventually lie in $V'$, and $x_i$ must eventually lie in $\pi(V)\cap\pi(V')$. Thus, eventually $\pi^{-1}(x_i)\cap V$ will be higher than $\pi^{-1}(x_i)\cap V'=\ol x_i$, which is a contradiction. 

\emph{Part II.} Now we show that if the multiplicity of each tangent cone $T_p X$ is at most $m<3/2$, then $X$ is a hypersurface, which will complete the proof. To this end let $f$ and $\Omega$ be as in Part I, and  define $g\colon \Omega\to\R$ by letting  $g(x)$ be the infimum of the height of $\pi^{-1}(x)\cap U'$ for any $x\in \Omega$. Then it follows that $g$ is also continuous by essentially repeating the argument in Part I given for $f$. Now it suffices to show that $f\equiv g$ on an open neighborhood of $o$. To see this  first recall that we chose $U'$ so that it intersects the $x_n$-axis, or $\pi^{-1}(o)$ only at $o$. Thus $f(o)=g(o)$. Now suppose, towards a contradiction, that there exists a sequence of points $x_i\in \Omega\smallsetminus\{o\}$ converging to $o$ such that $f(x_i)\neq g(x_i)$, and let $W_i:=B^{n-1}(x_i,r_i)$ be the balls of largest radii centered at $x_i$ such that $f\neq g$ everywhere on the interior of $W_i$, i.e., set $r_i$ equal to the distance of $x_i$ from the set of points where $f=g$. Then each $W_i$ will have a boundary point $o_i$ such that $f(o_i)=g(o_i)$. 
Note that since $o$ is not in the interior of $W_i$, $r_i\leq\|o-x_i\|$, so by the 
triangle
inequality:
$$
\|o-o_i\|\leq \|o-x_i\|+\|x_i-o_i\|=\|o-x_i\|+r_i\leq 2\|o-x_i\|.
$$
Thus, since $x_i$ converges to $o$, it follows that $o_i$ converges to $o$. Now if we set $\ol o_i:=(o_i, f(o_i))$, then it follows that $T_{\ol o_i} X$ converges to $T_o X=\R^{n-1}$, by the continuity of $f$ and $p\mapsto T_p X$. 
Further note that, if $\theta_i$ is the angle between $T_{o_i} X$ and $\R^{n-1}$, then, for any $A\subset\R^{n-1}$,
$$
\mathcal{H}^{n-1}\Big(T_{\ol o_i} X\cap \pi^{-1}(A)\Big)=\frac{1}{\cos(\theta_i)}\mathcal{H}^{n-1}(A).
$$
In particular, since $\theta_i\to 0$, we may choose for any $\epsilon>0$,  an index $k$ large enough so that
$$
\mathcal{H}^{n-1}\Big(T_{\ol o_k} X\cap \pi^{-1}\big(B^{n-1}(o_k,1)\big)\Big)< (1+\epsilon)\,\mathcal{H}^{n-1}\Big(B^{n-1}(o_k,1)\Big).
$$
In addition, since $T_{\ol o_k} X$ converges to $\R^{n-1}$, for any $\delta>0$ we may assume that $k$ is so large that
$$
T_{\ol o_k} X\cap\pi^{-1}\Big(B^{n-1}(o_k,1)\Big)\subset B^{n}(o_k,1+\delta).
$$
See Figure \ref{fig:cats}.
\begin{figure}[h]
   \centering
    \begin{overpic}[height=1.5in]{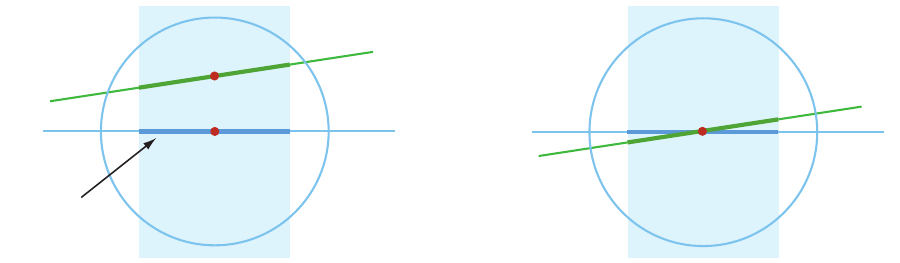} 
    \put(13,-3){$\pi^{-1}\big(B^{n-1}(o_k,1)\big)$}
    \put(71,-3){$\pi^{-1}(\B^{n-1})$}
    \put(23,12){$o_k$}
    \put(23,18){$\ol o_k$}
    \put(45, 13){$\R^{n-1}$}
    \put(43,23){$T_{\ol o_k}X $}
    \put(36, 7){$B^n(o_k,1+\delta)$}
     \put(-1,5){$B^{n-1}(o_k,1)$}
     \put(77.3,12){$o$}
     \put(90, 7){$(1+\delta)\B^n$}
    \end{overpic}
   \caption{}
   \label{fig:cats}
\end{figure}

Now fix $k$ and let us assume, after a translation, that $o:=\ol o_k$. Then $B^{n-1}(o_k,1)=\B^{n-1}$ and $B^{n}(o_k,1+\delta)=(1+\delta)\B^{n}$. So we may rewrite the last two displayed expressions as:
\begin{gather}
 \label{eq:measure}  \mathcal{H}^{n-1}\Big(T_{o} X\cap \pi^{-1}\big(\B^{n-1}\big)\Big)< 
(1+\epsilon)\,\mathcal{H}^{n-1}(\B^{n-1}),\\
\label{eq:measure1} T_{o} X\cap\pi^{-1}(\B^{n-1})\subset (1+\delta)\B^{n}.
\end{gather}
Further, if we set $W:=W_k$, then $o\in\partial W$, and  $f\neq g$ everywhere in the interior of $W$; therefore,  the projection  $\pi\colon X\to\R^{n-1}$ is at least $2$ to $1$ over the interior of $W$. So the  projections $\pi\colon X_{o,\lambda}\to W_{o,\lambda}$ are  at least $2$ to $1$ as well for all $\lambda\geq 1$. But note that $W_{o,\lambda}$ eventually fills half of $\B^{n-1}$ (since $o\in\partial W$). Consequently,  
\begin{equation}\label{eq:measure2}
\liminf_{\lambda\to\infty} \mathcal{H}^{n-1}\Big(X_{o,\lambda}\cap\pi^{-1}(\B^{n-1})\Big)\geq \frac{3}{2}\mathcal{H}^{n-1}(\B^{n-1}),
\end{equation} 
due to the basic fact that orthogonal projections, since they are uniformly Lipschitz,  do not increase the Hausdorff measure of a set, e.g., see \cite[Lemma 6.1]{falconer}. Also, here we have used the fact that, by Lemma \ref{lem:projection},  $\pi(X)$ contains an open neighborhood of $o$ and thus $\pi(X_{o,\lambda})=(\pi(X))_{o,\lambda}$ eventually covers $\B^{n-1}$ as $\lambda$ grows large. Combining \eqref{eq:measure2} with \eqref{eq:measure}, now yields that
\begin{eqnarray}\label{eq:measure3}
\liminf_{\lambda\to\infty} \mathcal{H}^{n-1}\Big(X_{o,\lambda}\cap\pi^{-1}(\B^{n-1})\Big)&\geq& \frac{3}{2(1+\epsilon)}\mathcal{H}^{n-1}\Big(T_{o} X\cap \pi^{-1}(\B^{n-1})\Big)\\ \notag
&\geq& \frac{3}{2(1+\epsilon)}\mathcal{H}^{n-1}\Big(T_{o} X\cap \B^n\Big),
\end{eqnarray}
since $\B^n\subset\pi^{-1}(\B^{n-1})$.
Next note that by \eqref{eq:measure1}
$$
T_oX\cap\pi^{-1}(\B^{n-1})+\delta\B^n
\subset (1+2\delta)\B^n.
$$
Further, by Lemma \ref{lem:coneconverge}, for sufficiently large  $\lambda$,
$$
X_{o,\lambda}\cap\pi^{-1}(\B^{n-1})\subset (T_oX+\delta\B^n)\cap\pi^{-1}(\B^{n-1})
\subset T_oX\cap\pi^{-1}(\B^{n-1})+\delta\B^n.
$$
Thus it follows that, for sufficiently large $\lambda$, $X_{o,\lambda}\cap\pi^{-1}(\B^{n-1})$ lies in $(1+2\delta)\B^n$, and so we may write
$$
X_{o,\lambda}\cap\pi^{-1}(\B^{n-1})\subset X_{o,\lambda}\cap(1+2\delta)\B^n.
$$
This in turn yields that
$$
\liminf_{\lambda\to\infty} \mathcal{H}^{n-1}\Big(X_{o,\lambda}\cap(1+2\delta)\B^{n}\Big)\geq\liminf_{\lambda\to\infty} \mathcal{H}^{n-1}\Big(X_{o,\lambda}\cap\pi^{-1}(\B^{n-1})\Big).
$$
The last inequality together with \eqref{eq:measure3} now shows
$$
\liminf_{\lambda\to\infty} \mathcal{H}^{n-1}\Big(X_{o,\lambda}\cap\B^{n}\Big)\geq \frac{3}{2(1+\epsilon)(1+2\delta)^{n-1}}\mathcal{H}^{n-1}\Big(T_{o} X\cap \B^n\Big).
$$
So, recalling that $\epsilon$ and $\delta$ may be chosen as small as desired, we conclude that the multiplicity of $T_{\ol o_k} X$ becomes arbitrarily close to $3/2$ as $k$ grows large. In particular it eventually exceeds any given constant $m<3/2$, which is a contradiction, as  we had desired.
\end{proof}

\section{Regularity of Sets with Positive Support and  Reach:\\ Proof of Theorem \ref{thm:main2} and Another Result}

\subsection{}The proof of the first half of Theorem \ref{thm:main2}, i.e., the $\C^1$ regularity of $X$, quickly follows from Theorem  \ref{thm:main1} via the next observation:

\begin{lem}\label{lem:posreach1}
Let $X\subset\R^n$ be a set with flat tangent cones and positive support.  Then the tangent cones of $X$ vary continuously.
\end{lem}
\begin{proof}
First note that the support ball $B_p$ of $X$ of radius $r$ at $p$ must also support $T_pX$ by Lemma \ref{lem:cone}. Thus, since $T_pX$ is a hyperplane, $B_p$ must be the unique support ball of radius $r$ for $X$ at $p$, up to a reflection through $T_pX$. Next note that if $p_i$ are a sequence of points of $X$ which converge to a point $p$ of $X$, then $B_{p_i}$ must converge to a support ball of $X$ at $p$ of radius $r$. Thus the limit of $B_{p_i}$ must coincide with $B_p$ or its reflection. Since $T_{p_i} X$ are tangent to $B_{p_i}$, it then follows that $T_{p_i}X$ must converge to the tangent hyperplane of $B_p$ at $p$, which is just $T_p X$. Thus $T_p X$ varies continuously.
\end{proof}

\subsection{}To prove the second half of Theorem \ref{thm:main2} which is concerned with the $\C^{1,1}$ regularity of $X$, we need the following lemma. A ball \emph{slides freely} in a convex set $K\subset\R^n$, if for some uniform constant $r>0$, there passes through each point of $\partial K$ a ball of radius $r$ which is contained in $K$. The proof of the following observation may be found in  \cite[p. 97]{hormander}:

\begin{lem} [H\"{o}rmander]\label{lem:hormander}
Let $K\subset\R^n$ be a convex set. If a ball slides freely inside $K$, then $\partial K$ is a $\C^{1,1}$ hypersurface.\qed
\end{lem}

Now the $\C^{1,1}$ regularity of $X$ in Theorem \ref{thm:main2}  follows by means of a M\"{o}bius transformation which locally sends hypersurfaces with double positive support to convex hypersurfaces in which a ball slides freely:

\begin{lem}\label{lem:inversion}
Let $X\subset\R^n$ be a hypersurface with flat tangent cones and double positive support by balls of uniform radius $r$. Let $B$ be any of these balls,  $p$ be the point of contact between $B$ and $X$, and  $i(X)$ be the inversion of $X$ with respect to $\partial B$. Then an open neighborhood  of $p$ in $i(X)$ lies on the boundary of a convex set in which a ball slides freely.
\end{lem}
\begin{proof}
By Lemma \ref{lem:posreach1}, tangent  hyperplanes of $X$ vary continuously. So by Theorem \ref{thm:main1}, $X$ is $\C^1$. Consequently, after replacing $r$ with a smaller value, and $X$ with a smaller neighborhood of $p$, we may assume that the interiors of the balls which lie on one side of $X$ never intersect the interiors of the balls which lie on the other side  (e.g., this follows from the tubular neighborhood theorem \cite{spivak:v1}). Furthermore, since $X$ is $\C^1$, there exists a continuous unit normal vector field $n\colon X\to\S^{n-1}$. After a translation, rescaling, and possibly replacing $n$ by $-n$, we may assume that  $r=1$, $B=\B^n$, and $c(p):=p+n(p)=o$. Now for each $x\in X$ let $B_x$ be the support ball of $X$ at $x$ with center at $c(x)$ and $B'_x$ be the other support ball of $X$ through $x$. Since $c$ is continuous, and $o$ is in the interior of $B_p$, there is an open neighborhood $U$ of $p$ in $X$  so that $o$ is in the interior of $B_x$ for all $x\in U$; see the diagram on the left hand side of Figure \ref{fig:balls}. 

\begin{figure}[h]
   \centering
    \begin{overpic}[height=2.1in]{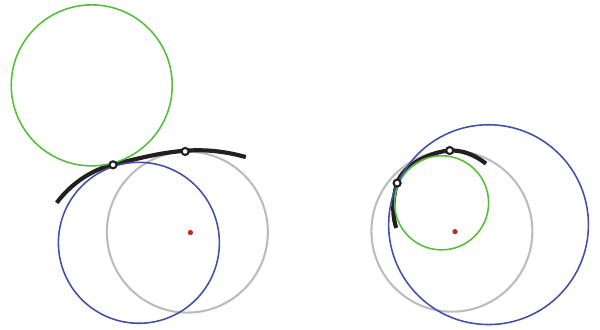} 
     \put(42,4){$\S^{n-1}$}
     \put(32,14){$o$}
     \put(30,32){$p$}
     \put(17,29){$x$}
     \put(3,6){$\partial B_x$}
     \put(-7,35){$\partial B'_x$}
     \put(95,3){$i(\partial B_x)$}
     \put(74,10){$i(\partial B'_x)$}
       \put(67,22){$i(x)$}
         \put(6,19){$U$}
    \end{overpic}
   \caption{}
   \label{fig:balls}
\end{figure}

Now consider $i(U)$ where $i\colon\R^{n}\smallsetminus\{o\}\to\R^n$ is inversion through $\partial B=\S^{n-1}$, i.e., $i(x):=x/\|x\|$. Note that if $S$ is any topological sphere which contains $o$ inside it, then $i$ maps the outside of $S$ to the inside of $i(S)$ (the ``inside" of $S$ is  the closure of the bounded component of $\R^n\smallsetminus S$, and the ``outside" of $S$ is the closure of the other  component). Thus, since $U$ lies outside $\partial B_x$, and $o$ is inside $\partial B_x$, it follows that $i(U)$ lies inside $i(\partial B_x)$; see the diagram on the right hand side of Figure \ref{fig:balls}. So through each point $i(x)$ of $i(U)$ there passes a ball bounded by $i(\partial B_x)$ which contains $i(U)$. Let $K$ be the intersection of all these balls. 
Then $K$ is a convex set and $i(U)\subset\partial K$.

Next recall that, by assumption, $B'_x$ lies outside $\partial B_y$ for all $x,y\in X$. Thus $i(B'_x)$ lies inside $i(\partial B_y)$, which in turn yields that $i(B'_x)\subset K$. Then since the radii of $i(B'_x)$ depend continuously on $x$, we may assume, after replacing $U$ by a possibly smaller open neighborhood of $p$, that the radii of $i(B'_x)$ are all 
greater than a uniform constant $r'>0$. Let $K'$ be the union of all balls of radius $r'$ in $K$. Then it is easy to check that $K'$ is a convex set with $i(U)\subset\partial K'$, and a ball (of radius $r'$) slides freely inside $K'$. So, by Lemma \ref{lem:hormander}, partial $K'$, and in particular $i(U)$ must be $\C^{1,1}$, which in turn yields that $U$ is $\C^{1,1}$, since $i$ is a local diffeomorphism, and finishes the proof.
\end{proof}

\begin{note}\label{note:lytchak}
In \cite[Prop. 1.4]{lytchak}, A. Lytchak proves that for a set $X$ of positive reach in a Riemannian manifold the following statements are equivalent: (i) $X$ is a $\C^{1,1}$ submanifold, (ii) $X$ is a topological manifold, and (iii) every tangent cone of $X$ is Euclidean. Furthermore it can be shown that for a topological hypersurface, having double positive support is equivalent to positive reach. Then, the second half of Theorem \ref{thm:main2} would follow from Lytchak's result. It appears however that our proof based on the inversion trick is more elementary or economical, which is why it has been retained in this work. See also \cite{bernig&lytchak} for other results of Bernig and Lytchak on tangent cones.
\end{note}

\subsection{} 
We close this section with one more regularity result for sets of positive reach which does not a priori assume that the boundary of the set is a hypersurface:

\begin{thm}\label{thm:posreach}
Let $K\subset\R^n$ be a set of positive reach. Suppose that the tangent cone of $K$ at each boundary point $p\in \partial K$ is a half-space. Then $\partial K$ is $\C^{1}$.
\end{thm}

This  result may be regarded as  another generalization of the fact that the boundary of a convex set with unique support hyperplanes is a $\C^1$ hypersurface. Indeed we use the inversion trick used in Lemma \ref{lem:inversion} to reduce the proof to the convex case:

\begin{lem}\label{lem:convexity}
Let $K\subset\R^n$ be a convex set with interior points, $p\in\partial K$, and suppose that there exists an open neighborhood $U$ of $p$ in $\partial K$ such that through each point $q$ of $U$ there passes a unique support hyperplane of $K$. Then $U$ is $\C^1$.
\end{lem}

The above fact follows from Theorem \ref{thm:main2}, since for any convex set $K\subset\R^n$ with interior points, $\partial K$ is a hypersurface (e.g. see \cite[p. 3]{busemann}), and has positive support; further, by Lemma \ref{lem:convexsurface} below, $T_p(\partial K)=\partial (T_p K)$ which shows that $T_p(\partial K)$ is flat. We also include here a more analytic proof using basic properties of convex functions:

\begin{proof}
After a translation we may assume that $p=o$ and the support hyperplane of $K$ at $p$ coincides with the hyperplane of the first $(n-1)$ coordinates. Then, assuming it is sufficiently small, $U$ coincides with the graph of a convex function $f\colon \Omega\to\R$ for some convex neighborhood $\Omega$ of $o$ in $\R^{n-1}$. To show that $f$ is $\C^1$, it suffices to check that it is differentiable at each point of $\Omega$ \cite[Thm 1.5.2]{schneider:book}.
Further, $f$  is differentiable at a point if and only if it has only one subgradient at that point \cite[Thm 1.5.12]{schneider:book}. But the subgradient is unique at a point of $\Omega$ if and only if the epigraph of $f$ has a unique normal at the corresponding point \cite[Thm 1.5.11]{schneider:book}, which is  the case due to uniqueness of the support hyperplanes. 
\end{proof}

Now we are ready to prove the last result of this section:

\begin{proof}[Proof of Theorem \ref{thm:posreach}]
First recall that through each point $p\in\partial K$ there passes a support ball $B_p$ of radius $r$ for some uniform constant $r>0$.
Since $T_p K$ is a half-space, it follows that $B_p$ is  the unique support ball of $K$ of radius $r$ passing through $p$, since the interior of $B_p$ must be disjoint from $T_p K$ by Lemma \ref{lem:cone}. Further note that if $p_i\in\partial K$ are any sequence of points converging to $p$, then any limit point
of the sequence of balls $B_{p_i}$  will be a support  ball of $K$ at $p$ of radius $r$, which must be $B_p$ due to the uniqueness of $B_p$; therefore, $B_{p_i}$ converges to the support ball at $p$. So $p\mapsto B_p$ is continuous.

Now fix a point $p\in\partial K$. After a translation and rescaling we may assume that $B_p=\B^n$. In particular, the center of $B_p$ is $o$. By continuity of the support balls there exists an open ball $V$ centered at $p$ in $\R^n$ such if we set 
$$
U:=V\cap\partial K,
$$
then $B_x$ contains $o$ for every $x\in U$. Next consider the inversion $i$ with respect to $\partial B_p=\S^{n-1}$. We claim that if 
$$
K':=K\cap \ol V,
$$
where $\ol V$ denotes the closure of $V$, then $i(K')$ is convex. Indeed we show that through each point of $\partial (i(K'))=i(\partial K')$ there passes a sphere which contains $i(K')$.
To see this note that if $x\in i(\partial K')$, then either  $x\in i(\partial  V)$ or  $x\in i(V)$. If  $x\in i(\partial  V)$, then note that $i(\partial  V)$ bounds the ball $i(\ol V)$ which contains $i(K')$, since $\ol V$ contains $K'$. If, on the other hand, $x\in  i(V)$,  note that
$$
i(\partial K')\cap i(V) = i(\partial K'\cap V) = i(\partial K\cap V)=i(U),
$$
and recall that, as we already showed in the proof of Lemma \ref{lem:inversion},  through each point $i(x)$ of $i(U)$ there passes a sphere $i(\partial B_x)$ which contains $i(K)$ and therefore $i(K')$ inside it. So  we conclude that $i(K')$ is convex. 
Next note that for all $x\in \partial K'$
\begin{equation}\label{eq:tix}
T_{i(x)}i(K')=di_x(T_x K'),
\end{equation}
where $di_x$ is the differential map or the jacobian of $i$ at $x$, which has full rank for all $x\in\R^n\smallsetminus\{o\}$ which includes $U$. Furthermore, if $x\in U$, then $T_x K'=T_x K$ which is a half space, and so $di_x(T_x K')$ is a half-space as well, since $i_x$ has full rank. Thus \eqref{eq:tix} shows that $T_{i(x)}i(K')$ is a half-space for all $x\in U$. So $i(K')$ has a unique support hyperplane at all points of $i(U)$. But note that $i(U)$ is open in $i(\partial K')$;
because,  $i\colon \partial K'\to \partial K$ is a homeomorphism (and therefore must preserve open sets by the theorem on the invariance of domain). Now it follows from Lemma \ref{lem:convexity} that $i(U)$ is $\C^1$, which in turn yields that so is $U$, since $i$ is a local diffeomorphism. 
\end{proof}

\section{Tangent Cones of Real Analytic Hypersurfaces: Proof~of~Theorem~\ref{thm:main3}}\label{sec:tangents}

Theorem \ref{thm:main3} follows quickly from Theorem \ref{thm:conesemicone} proved in Section \ref{subsec:3cones} below which relates the three notions of tangent cones we discussed earlier in Section \ref{sec:prelim}. To prove Theorem \ref{thm:conesemicone}, we first  need to study in Section \ref{subsec:sign} the sign of real analytic functions in the neighborhood of their zero sets.

\subsection{}\label{subsec:sign}
We say that a function $f\colon\R^n\to\R$ changes sign at a point $p\in\R^n$ provided that in every open neighborhood  of $p$ in $\R^n$ there are points where $f>0$ and $f<0$. If $f$ changes sign at every point of a set $X\subset\R^n$, then we say that $f$ changes sign on $X$. The following observation shows that in certain cases we may assume that the function defining an analytic set changes sign on that set:

\begin{prop}\label{prop:sign}
Let  $X\subset\R^n$ be a hypersurface which coincides with the zero set $Z(f):=f^{-1}(0)$ of an analytic function $f\colon\R^n\to\R$. Then for every point $p\in X$ there exists an open neighborhood $U$ of $p$ in $\R^n$, and an analytic function $g\colon U\to\R$ such that $Z(g)=X\cap U$, and $g$ changes sign on $X\cap U$.
\end{prop}

To prove this result, let us suppose that $p=o$,  and let $\C^\omega_o$ denote the ring of germs of analytic functions at $o$. It is well-known that $\C^\omega_o$ is Noetherian and is a unique factorization domain, e.g., see \cite[p. 148]{ZP}. 
So $f$  is the product of finitely may irreducible factors $f_i$ in $\C^\omega_o$ (abusing notation, we denote here the functions and their germs by the same symbols). Further, note that $Z(f)$ is the union of $Z(f_i)$, and $\dim(Z(f))=n-1$ by assumption. Thus, by {\L}ojasiewicz's structure theorem for real analytic varieties
 \cite[Thm. 6.3.3 (c)]{krantz&parks:book},
we may assume that there exists a factor $g$ of $f$ such that
\begin{equation}\label{eq:zg}
\dim \big(Z(g)\big)=n-1.
\end{equation} 
Indeed, each $Z(f_i)$  is an analytic set which admits a stratification into disjoint union of finitely many smooth manifolds or \emph{strata} by {\L}ojasiewicz's theorem. The maximum dimension of these strata then defines $\dim(Z(f_i))$. So, since $Z(f_i)\subset Z(f)$, $\dim(Z(f_i))\leq n-1$. Suppose towards a contradiction that $\dim(Z(f_i))<n-1$ for all $i$.
Then since each strata of $Z(f_i)$ is  smooth,  $\mathcal{H}^{n-1}(Z(f_i))=0$. This  in turn yields that $\mathcal{H}^{n-1}(Z(f))=0$ which is not possible.

Now \eqref{eq:zg} implies that 
 $g$ satisfies the equation \eqref{eq:null} below which is a type of real nullstellensatz \cite{BCR,risler,ruiz}.
Here $(g)\subset \C^\omega_o$ is the ideal generated by $g$, i.e., the collection of all germs $\phi g$ where $\phi\in \C_o^\omega$. Further $I(Z(g))\subset \C^\omega_o$ is the ideal of germs in $\C^\omega_o$ which vanish on $Z(g)$.

\begin{lem}
Let $g\in\C_o^\omega$ be irreducible, and suppose that $\dim(Z(g))=n-1$. Then 
\begin{equation}\label{eq:null}
(g)=I\big(Z(g)\big).
\end{equation}
\end{lem}
\begin{proof}
See  the proof of Theorem 4.5.1 in \cite[p. 95]{BCR}.  Although that theorem is stated in the ring of polynomials, as opposed to our present setting of $C_o^\omega$, the proof of the implication $(v)\Rightarrow (ii)$ in that result depends only on general algebraic properties of ideals in Noetherian rings and thus applies almost verbatim to the present setting. In particular, irreducibility of $(g)$ implies that $(g)$ is a prime ideal of height one. On the other hand $(g)\subset I(Z(g))$, and \eqref{eq:zg} implies that $I(Z(g))$ also has height one in $C_o^\omega$. Thus it follows that $(g)=I(Z(g))$.
\end{proof}

 Now it follows  that the gradient of $g$ cannot vanish  identically  on $Z(g)$; because otherwise, $\partial g/\partial x_i\in I(Z(g))=(g)$ which yields that $$
 \frac{\partial g}{\partial x_i}=\phi_i g
 $$ 
 for some  $\phi_i\in\C^\omega_o$. Consequently, by the product rule, all partial derivatives of  $g$ of any order must vanish at $o$. But then, since $g$ is analytic, it must vanish identically on $U$ which is not possible by \eqref{eq:zg}. So we may assume that $g$ has a regular point  in $Z(g)\subset X\cap U$. Then $g$ must assume different signs  on $U$, and therefore $U\smallsetminus Z(g)$ must be disconnected. This, we claim, implies that $Z(g)=X\cap U$ which would complete the proof. To establish this claim, first note that since $g$ is a factor of $f$, we already have 
 $$
 Z(g)\subset X\cap U.
 $$
  Further, by the generalized Jordan-Brouwer separation theorem for closed sets in $\R^n$ which are locally homeomorphic to $\R^{n-1}$, or more specifically Alexander duality, we may choose 
 $U$  so that $U\smallsetminus X$ has exactly two components, e.g., see \cite[8.15]{dold}. Using some elementary topology, one can then show that if $Z(g)$ is a proper subset of $X\cap U$, then it cannot separate $U$ which is a contradiction. It would then follow that 
 $$
 Z(g)=X\cap U
 $$
  which would complete the proof of Proposition \ref{prop:sign}. So all that remains is to check that if $X$ separates $U$ into a pair of components, and $A$ is any proper subset of $X\cap U$, then $A$ cannot separate $U$. To see this let $U_1$, $U_2$ be the two components of $U\smallsetminus X$, and set
$V := U_1 \cup U_2 \cup A'$,
where $A':=(X\cap U)\smallsetminus A$.
We show that $V$  is connected which is all we need.  To see this let $W\subset V$ be a nonempty set which is both open and closed in $V$.  Then 
 $W\not\subset A'$, for otherwise there would be an open set $O\subset \R^n$ with
$O\cap V =  W \subset A' $.  This would  imply that $O$ is disjoint from
both $U_1$ and $U_2$, which is impossible since $U_1 \cup U_2=U\smallsetminus X$
is dense in $U$.   We may assume then that $W$ meets $U_1$. Then, since $U_1$ is connected, $U_1\subset W$.  Further, since $W$ is
closed, it will contain the closure of $U_1$ in $V$ which is $U_1\cup A'$.
But $U_1\cup A'$ is  not open in $V$ (for any neighborhood
of a point of $A'$ meets both $U_1$ and $U_2$).  Thus $W\neq U_1\cup A'$, and therefore it must contain a point of $U_2$.
Then, since $U_2$ is connected,  $U_2 \subset W$, which implies that
$W = U_1 \cup U_2 \cup A =V$.  So $V$ is connected.

\subsection{}\label{subsec:3cones}
Now we proceed towards proving Theorem \ref{thm:conesemicone} below which shows, via Proposition \ref{prop:sign} above, that  
if the tangent cone of an analytic hypersurface is a hypersurface, then it is symmetric.
Let $U\subset\R^n$ be an open neighborhood of $o$, 
$f\colon U\to\R$ be a  $\C^{k\geq 1}$ function with $f(o)=0$, and suppose that $f$ does not vanish to order $k$  at $o$.  Then by Taylor's theorem
$$
f(x)=h_f(x)+r_f(x)
$$
where $h_f(x)$
is a nonzero homogenous polynomial of degree $m$, i.e.,
$$
h_f(\lambda x)=\lambda^m h_f(x),
$$
for every $\lambda\in\R$, and $r_f\colon\R^n\to\R$ is a
continuous function  which satisfies
$$
\lim_{x\to0}|x|^{-m}r_f(x)=0.
$$
 
Now recall that $Z(f):=f^{-1}(0)$,
and $\tilde T_o Z(f)$ denotes the symmetric tangent cone, i.e., the limit of all sequences of secant \emph{lines} through $o$ and $x_i\in Z(f)\smallsetminus\{o\}$ as $x_i\to o$. 
Also let $Z(h_f):=h_f^{-1}(0)$ be the zero set of $h_f$. Then we have:
\begin{lem}\label{lem:1}
$\tilde T_o Z(f)\subset Z(h_f)$. 
\end{lem}
\begin{proof}
Suppose $v\in \tilde T_o Z(f)$. Then, it follows from Lemma \ref{lem:cone} that there are points $x_i\in Z(f)\smallsetminus\{o\}$, and numbers $\lambda_i\in\R$ such that $\lambda_ix_i\to v$. Since $x_i\in Z(f)$, 
$$
0=f(x_i)=h_f(x_i)+r_f(x_i)
$$
which yields that
$$
0=|x_i|^{-m} \big(h_f(x_i)+r_f(x_i)\big)=h_f\big(|x_i|^{-1}x_i\big)+|x_i|^{-m}r_f(x_i).
$$
Consequently
$$
0=\lim_{x_i\to o} h_f\big(|x_i|^{-1}x_i\big)+0=h_f\left(\lim_{x_i\to o} |x_i|^{-1}x_i\right).
$$
But,
$$
\lim_{x_i\to o} |x_i|^{-1}x_i=\lim_{x_i\to o} |\lambda_ix_i|^{-1}|\lambda_i|x_i=\pm |v|^{-1}v.
$$
So we conclude
$$
0=h_f\big(|v|^{-1}v\big)=|v|^{-m}h_f(v),
$$
which shows that $h_f(v)=0$, or $v\in Z(h_f)$. 
\end{proof}

In contrast to the above lemma, in general $Z(h_f)\not\subset \tilde T_o Z(f)$. Consider for instance the case where $f(x,y)=x(y^2+x^4)$. Then $Z(f)$ is just the $y$-axis, while $h_f(x,y)=xy^2$, and therefore $Z(h_f)$ is both the $x$-axis and the 
$y$-axis. So in order to have $Z(h_f)\subset  \tilde T_o Z(f)$, we need additional conditions, as given for instance by the next lemma. Recall that a function $f\colon\R^n\to\R$ \emph{changes sign} on a set $X\subset\R^n$, provided that for every point $x\in X$, and open neighborhood $U$ of $x$ in $\R^n$ there are points in $U$ where $f>0$ and $f<0$. Let $\tilde Z(h_f)\subset Z(h_f)$ be the set of points $p$ where $h_f$  changes sign at $p$. Then we show that:

\begin{lem}\label{lem:2}
 $\tilde Z(h_f)\subset \tilde T_o Z(f)$.
\end{lem}
\begin{proof}
Suppose, towards a contradiction, that there is $v\in \tilde Z(h_f)$ such that $v\not\in \tilde T_o Z(f)$. Then it follows from Lemma \ref{lem:cone} that there exists an open neighborhood $U$ of $v$ in $\R^n$ and an open ball $B$ centered at $o$ such that $\cone(U)\cap B\cap Z(f)=\{o\}$, where $\cone(U)$ is the set of all lines which pass through $o$ and points of $U$. So we have $f\neq 0$ on $\cone(U)\cap B\smallsetminus\{o\}$.
Consequently if we set
\begin{equation}\label{eq:flambda}
f_\lambda(x):=\lambda^mf(\lambda^{-1}x),
\end{equation}
then it follows that $f_\lambda\neq 0$ on $\cone(U)\cap B\smallsetminus\{o\}$ for $\lambda\geq 1$. But note that,
by homogeneity of $h_f$,
$$
f_\lambda(x)=\lambda^mh_f(\lambda^{-1}x)+\lambda^m r_f(\lambda^{-1}x)
=h_f(x)+\lambda^m r_f(\lambda^{-1}x),
$$
which yields that
\begin{equation}\label{eq:fp}
\lim_{\lambda\to \infty}f_\lambda(x)=h_f(x).
\end{equation}
Furthermore, by assumption, there are points in $U$ where $h_f>0$ and $h_f<0$. Consequently, for large $\lambda$, $f_\lambda$ must change sign on $U$ as well, which in turn implies that $f_\lambda=0$ at some point of $U$, and we have a contradiction.
\end{proof}

Now let  $T_oZ(f)\subset \tilde T_o Z(f)$ denote as usual the tangent cone of $Z(f)$, i.e.,  the limit of all sequences of secant \emph{rays} (as opposed to lines) which emanate from $o$ and pass through $x_i\in Z(f)\smallsetminus\{o\}$ as $x_i\to o$. 

\begin{thm}\label{thm:conesemicone}
Let $U\subset\R^n$ be an open neighborhood of $o$ and $f\colon U\to\R^n$ be a $\C^{k\geq 1}$ function with $f(o)=0$ which does not vanish to order $k$ at $o$. Suppose that $Z(f)$ is homeomorphic to $\R^{n-1}$, $f$ changes sign on $Z(f)$, and $T_oZ(f)$ is also a hypersurface. Then 
$$
\tilde T_oZ(f)=\tilde Z(h_f)=T_oZ(f).
$$
In particular, $T_oZ(f)$ is symmetric with respect to $o$, i.e., $T_oZ(f)=-T_oZ(f)$.
\end{thm}
\begin{proof}
For convenience we may assume that $U=\R^n$. Further note that, 
since $(T_o Z(f))^*=-T_o Z(f)$, we have
$$\tilde T_o Z(f)=T_oZ(f)\cup-T_oZ(f).$$
So it is enough to show that $T_oZ(f)=\tilde Z(h_f)$; because then $-T_oZ(f)=-\tilde Z(h_f)=\tilde Z(h_f)$ by homogeneity of $h_f$; consequently,  
$\tilde T_o Z(f)\subset\tilde Z(h_f)$, which in turn yields that $\tilde T_o Z(f)=\tilde Z(h_f)$ by Lemma \ref{lem:2}.

By Lemma \ref{lem:2}, $\tilde Z(h_f)\subset T_oZ(f)$. So it remains to show that $T_oZ(f)\subset\tilde Z(h_f)$, i.e., we have to check that
 $h_f$ changes sign on $T_oZ(f)$. 
To see this, for $\lambda\in\R^+$, let $Z(f)_\lambda:=Z(f)_{o,\lambda}$ be the set of points $\lambda x$ where $x\in Z(f)$. Recall that by Lemma 
\ref{lem:coneconverge}, for large $\lambda$ we have
\begin{equation}\label{eq:lim}
Z(f)_\lambda\cap B^n(o,r)\subset\big(T_oZ(f)+\epsilon \B^n\big)\cap B^n(o,r).
\end{equation}
Next note that, if $f_\lambda$ is given by \eqref{eq:flambda}, then 
$$
f^{-1}_\lambda(0)=\{x\mid\lambda^mf(\lambda^{-1} x)=0\}
=\{x\mid f(\lambda^{-1} x)= 0\}
=\{\lambda y\mid f(y)= 0\}
=Z(f)_\lambda.
$$
By the generalized Jordan-Brouwer separation theorem \cite[Sec. 8.5]{dold}, $\R^n\smallsetminus Z(f)_\lambda$ has precisely two components: $(\R^n\smallsetminus Z(f)_\lambda)^\pm$. We may suppose that $f_\lambda>0$ on $(\R^n\smallsetminus Z(f)_\lambda)^+$ and $f_\lambda<0$ on $(\R^n\smallsetminus Z(f)_\lambda)^-$. It follows from \eqref{eq:lim} that inside any ball $B^n(o,r)$  we have $(\R^n\smallsetminus Z(f)_\lambda)^\pm$ converging, with respect to the Hausdorff topology, to the components of $\R^n\smallsetminus T_oZ(f)$ which we denote by $(\R^n\smallsetminus T_oZ(f))^\pm$ respectively:
$$
\Big(\R^n\smallsetminus Z(f)_\lambda\Big)^\pm\cap B^n(o,r)\longrightarrow \Big(\R^n\smallsetminus T_oZ(f)\Big)^\pm\cap B^n(o,r).
$$
In particular, if $x\in  (\R^n\smallsetminus T_oZ(f))^+$, then eventually (as $\lambda$ grows large) $x\in  (\R^n\smallsetminus Z(f)_\lambda)^+$, and thus $f_\lambda(x)>0$. Consequently,  $h_f(x)\geq 0$ by \eqref{eq:fp}, and we conclude that $h_f\geq 0$ on  $(\R^n\smallsetminus T_oZ(f))^+$. Similarly, we have 
$h_f\leq 0$ on  $(\R^n\smallsetminus T_oZ(f))^-$. So, since $h_f$ cannot vanish identically on any open set, it follows that $h_f$ changes sign on $T_oZ(f)$ as desired.
\end{proof}

\subsection{ }
Theorem \ref{thm:conesemicone} together with Proposition \ref{prop:sign} immediately yield:

\begin{cor}\label{cor:thmprop}
Let $X\subset\R^n$ be a real analytic hypersurface, and suppose that $T_p X$ is also a hypersurface for all $p\in X$, then each $T_p X$ is symmetric with respect to $p$, i.e., $T_p X=(T_p X)^*$. \qed
\end{cor}

Equipped with  the last observation, we are now ready to prove the main result of this section:

\begin{proof}[Proof of Theorem \ref{thm:main3}]
Since through each point $p$ of $X$ there passes a support ball $B$, it follows from Lemma \ref{lem:cone} that each $T_p X$ must lie in a half-space  whose boundary $H_p$ is tangent to $B$ at $p$. Then Corollary \ref{cor:thmprop} implies that  $T_p X\subset H_p$. But $T_p X$ is open in $H$, by the theorem on the invariance of domain,
and $T_pX$ is closed in $H$ since it is closed in $\R^n$ by Lemma \ref{lem:cone}. Thus $T_p X=H_p$. So we conclude that each $T_p X$ is flat. Then, by Lemma \ref{lem:posreach1}, $T_p X$ depends continuously on $p$, and we may apply Theorem \ref{thm:main1} to conclude that $X$ is $\C^1$. 

Next  note that if $X$ is any convex hypersurface, then it has positive support; further, its tangent cones are automatically hypersurfaces, by Lemma \ref{lem:convexsurface} below. So by the above paragraph $X$ is $\C^1$. 
\end{proof}

\begin{lem}\label{lem:convexsurface}
Let $K\subset\R^n$ be a closed convex set with interior points, and $p\in\partial K$. Then $T_p K$ is a convex set with interior points, and 
$$\partial (T_p K)=T_p(\partial K).$$ In particular $T_p(\partial K)$ is a hypersurface.
\end{lem}
\begin{proof}
We may suppose that $p=o$. Then it follows from Lemma \ref{lem:coneconverge} that
$$
T_p K=\bigcup_{\lambda\geq 0} \lambda K,
$$
once we note that, since $K$ is convex, $\lambda_1K\subset\lambda_2 K$ whenever $\lambda_1\leq \lambda_1$. Thus $T_p K$ is convex, and obviously has interior points since $K\subset T_p K$. This inclusion also shows that $\partial K\subset T_p K$, since tangent cones are always closed, by Lemma \ref{lem:cone}, and so 
$$
T_p(\partial K)\subset T_p(T_pK)=T_p K.
$$
Now suppose towards a contradiction that $T_p(\partial K)$ contains a ray $\ell$ which lies in the interior of $T_p K$. Then, since $T_p K$ is convex, there exists a cone $C$ about $\ell$ which lies in $T_p K$. Now since $\lambda C=C$, it follows that $C\cap B\subset K$ for some ball $B$ centered at $p$. In particular, after making $C$ smaller, we may assume that $C\cap B$ intersects $\partial K$ only at $p$. Hence $\ell$ cannot belong to $T_p(\partial K)$ by Lemma \ref{lem:cone}. So we conclude that $T_p(\partial K)=\partial(T_p K)$.
\end{proof}

\subsection{}\label{subsec:cusp}
Here we show that 
when $n=2$ in Theorem \ref{thm:main3}, it is not necessary to assume that the hypersurface $X$ have positive support. To see this let
 $\Gamma:=X\subset\R^2$ be a real analytic hypersurface or \emph{simple curve}, i.e., suppose that each point $p\in\Gamma$ has an open neighborhood $U$ homeomorphic to $\R$. Then we call the closure of each component of $U\smallsetminus \{p\}$ in $U$, which we denote by $b_1$ and $b_2$,  a \emph{half-branch} of $\Gamma$ at $p$. By the ``curve selection lemma" \cite{milnor:singular},  which also holds for semianalytic sets \cite[Prop. 2.2]{bv}, there exist $\C^1$ curves $\gamma_i \colon[0,1)\to b_i$ with $\gamma_i(0)=p$ such that $\gamma_i'\neq 0$; see \cite[p. 956]{OW}.  Thus each half-branch $b_i$ has a well-defined tangent ray $\ell_i:=\{p+t(\gamma_i)'_+(0)\mid t\geq 0\}$ emanating from $p$, where $(\gamma_i)'_+$ denotes the right hand derivative. It follows then from Corollary \ref{cor:outer} that $T_p\Gamma=\ell_1\cup \ell_2$. If $T_p \Gamma$ consists of only one ray, i.e., $\ell_1=\ell_2$, then we say that $\Gamma$ has a \emph{cusp} at $p$. 
Otherwise, $T_p\Gamma$ is a simple curve. Consequently, by Corollary \ref{cor:thmprop}, $T_p \Gamma$ must be flat, i.e, $\dir(\ell_1)=-\dir(\ell_2)$ which in turn yields that $U=b_1\cup b_2$ is $\C^1$. So we obtain the following special case of Theorem \ref{thm:main3} for $n=2$:

\begin{cor}\label{cor:cusp}
Let $\Gamma\subset\R^2$ be a  real analytic  simple curve. Then either $\Gamma$ has a cusp, or else it is $\C^1$.
\qed
\end{cor}

When, on the other hand,  $n\geq 3$ in Theorem \ref{thm:main3}, the positive support assumption cannot in general be abandoned (Example \ref{ex:analytic}). Further, note that the above corollary implies that if the tangent cones of a simple curve are  simple curves, then they must be flat. This again does not generalize to higher dimensions (Example \ref{ex:fermat}).
Finally we should mention that Corollary \ref{cor:cusp}  also follows from  classical resolution of  singularities;  see  Appendix \ref{sec:appendix}, specifically Corollary \ref{cor:appendix}.

\section{Regularity of Real  Algebraic Convex Hypersurfaces:\\Proof of Theorem \ref{thm:main4}}
\subsection{} First we show that $X$ in Theorem \ref{thm:main4} is an entire graph. To this end we need to employ the notion of the \emph{recession cone}  \cite{rockafellar} of a closed convex set $K\subset\R^n$ which is defined as
$$
\rc(K):=\{x\in\R^n\mid K+x\subset K\}.
$$
Further, let  $\nc(K)$ be the \emph{normal cone} \cite{rockafellar} of $K$ which is defined as the set of all outward normals  to support hyperplane of $K$. The following observation is implicit in  \cite[Thm. 2]{wu:spherical}, see also \cite[Prop. 3.1]{ghomi:deformation}.

\begin{lem}[Wu \cite{wu:spherical}]
 If $K\subset\R^n$ is a closed convex set with interior points and boundary $\partial K$  homeomorphic to $\R^{n-1}$, then
$
\rc(K)\cap\S^{n-1}\cap(-\nc(K))\neq\emptyset. 
$\qed
\end{lem}

By assumption, the hypersurface $X$ in Theorem \ref{thm:main4} coincides with  $\partial K$ for  some convex set $K$ satisfying the hypothesis of the above lemma. We show that for any unit vector $u\in\rc(K)\cap(-\nc(K))$,  $\partial K$  is an entire graph in the direction $u$, i.e., it intersects every line parallel to $u$.

We may suppose for convenience that $u=(0,\dots,0,1)$,  and  $K$ lies  in the half-space $x_n\geq 1$. Let $\H^n$ be the open half-space $x_n> 0$ and consider the projective transformation $P\colon \mathbf{H}^n\to \R^n$ given by
$$
P(x_1,\dots, x_n):=\left (\frac{x_1}{x_n},\dots, \frac{x_{n-1}}{x_n}, \frac{1}{x_n}\right).
$$
Note that $P$ preserves line segments, and so maps convex sets to convex sets. In particular, the closure $\ol{P(K)}$ will be a closed convex set.  Next let $f\colon\R^{n}\to\R$ be the algebraic function with $\partial K=f^{-1}(0)$ and suppose that $f$ has degree $d$. Then 
$$\ol f(x):=(x_n)^d f(P(x)),$$
is an algebraic function on $\R^n$,  and it is not hard to  check that
\begin{equation}\label{eq:pkf}
\partial \ol{P(K)}=\ol{f}^{-1}(0).
\end{equation}
So by Theorem \ref{thm:main3}, $\partial \ol{P(K)}$ is $\C^1$, since it is a real analytic convex hypersurface.

Next note that $o\in\partial \ol{P(K)}$, since by assumption $K$ contains the ray $(0,\dots,0,t)$ where $t\geq 1$. Since $\partial \ol{P(K)}$ is a $\C^1$ convex hypersurface supported by $\{x_n=0\}$, it follows  that each  ray $\ell$ given by $tv$ where $v$ is a unit vector with $\langle v, u\rangle>0$ and $t\geq 0$,  intersects $\partial \ol{P(K)}\cap \H^n=\partial P(K)=P(\partial K)$ in exactly one point, see Figure \ref{fig:projtran}. In particular,
$$
P(\partial K)\cap\ell\neq\emptyset.
$$
Now note that $P^2$ is the identity map on $\H^n$.  Thus the last expression yields that
$$
\partial K\cap P(\ell)=P^2(\partial K)\cap P(\ell)=P\big(P(\partial K)\cap \ell\big)\neq\emptyset.
$$
It remains only to note that $P(\ell)$ is parallel to $(0,\dots,0,1)$, and intersects $\{x_n=1\}$ at the same point as does $\ell$, which completes the proof that $\partial K$ is an entire graph.
 
 \begin{figure}[h]
   \centering
    \begin{overpic}[height=1.75in]{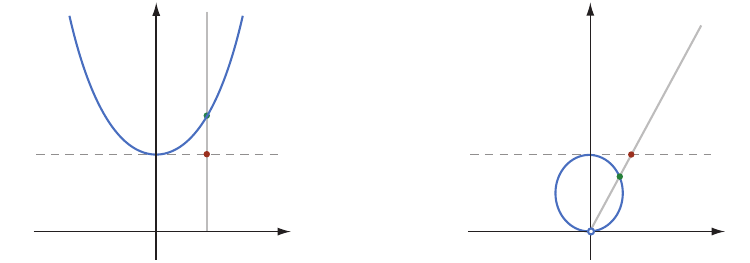} 
    \put(38,14){$x_n=1$}
     \put(5,23){$\partial K$}
      \put(21.25,26){$P(\ell)$}
       \put(64,7){$P(\partial K)$}
        \put(92,26){$\ell$}
    \end{overpic}
   \caption{}
   \label{fig:projtran}
\end{figure}

 \subsection{}\label{subsec:projtranform}
 Now we show that if $X=\partial K$ is strictly convex, then its projective closure is a $\C^1$ hypersurface in $\RP^n$. 
 First, let us recall that the standard embedding $i\colon\R^n\to\RP^n$ is given by
 $$
 (x_1,\dots,x_n)\overset{i}{\longmapsto} [x_1:\dots:x_{n}:1],
 $$
 where $[x_1:\dots:x_{n+1}]$ are the homogeneous coordinates of $\RP^n$. Then the projective closure of  $X$ is  defined as $\overline{i(X)}$, i.e., the closure of $i(X)$ in $\RP^n$. To show that $\overline{i(X)}$ is $\C^1$, we need to establish that the closure $\ol{P(K)}$ is compact, which  in turn is due to the following basic fact:
 
 \begin{lem}\label{lem:rc}
 Let $K\subset\R^n$ be a convex set. Suppose that $K$ is contained between a pair of hyperplanes $H_0$ and $H_1$, and $H_0\cap K$ is compact and nonempty. Then $K$ is compact.
 \end{lem}
 \begin{proof}
 First note that $K$ is compact if and only if it does not contain any rays, or equivalently $\rc(K)=\emptyset$. So, if $K$ is not compact, then there exists a ray $\ell$ in $K$ which emanates from a point $p$ of $K$. Let $H$ be the hyperplane passing through $p$ which is parallel to $H_0$. Then $\ell\in\rc(K\cap H)$. Further note that since $H$ and $H_0$ are parallel, $\rc(H)=\rc(H_0)$. So
 $$
 \rc(K\cap H)=\rc(K)\cap\rc(H)=\rc(K)\cap\rc(H_0)=\rc(K\cap H_0)\neq\emptyset,
 $$
 which is a contradiction since $K\cap H$ is compact.
 \end{proof}
 
 Now note that if $\partial K$ is strictly convex, and as in the previous subsection we assume that $\partial K$ is supported below by the hyperplane $x_n=1$, then this hyperplane intersects $\partial K$ and therefore $K$ at only one point (which is compact). Consequently $\ol{P(K)}$ will intersect $\{x_n=1\}$ only at one point as well, since $P$ is the identity on $\{x_n=1\}$. Thus, since $\ol{P(K)}$ is contained between $\{x_n=0\}$ and $\{x_n=1\}$, it follows from Lemma \ref{lem:rc} that $\ol{P(K)}$ is compact. So $\ol{P(K)}\cap\{x_n=0\}$ is compact. But recall that $\partial \ol{P(K)}$ is algebraic, which yields that so is $\ol{P(K)}\cap\{x_n=0\}$. Consequently, $\ol{P(K)}\cap\{x_n=0\}$ may not contain any line segments, so it must consist of only a single point since $\ol{P(K)}\cap\{x_n=0\}$ is convex.
 This implies that 
 $$
 \partial\ol{P(K)}=\ol{P(\partial K)}.
 $$
  So $\ol{P(\partial K)}$ is a $\C^1$ hypersurface by \eqref{eq:pkf}. Further, since $\ol{P(\partial K)}$ is compact, it coincides with its own projective closure. Thus, the projective closure of $\ol{P(\partial K)}$ is a $\C^1$ hypersurface. This  yields that the projective closure of $\partial K$ must be $\C^1$ as well, due to the commutativity of the following diagram:
 $$
 \begin{CD}
 \RP^n @>\tilde{P}>> \RP^n\\
 @AAiA                       @AAiA\\
 \H^n @>P>> \R^n
 \end{CD}
$$
where, in terms of the homogeneous coordinates $[x_1:\dots:x_{n+1}]$,  $\tilde P$ is the map
$$
[x_1:\dots:x_n:x_{n+1}]\overset{\tilde P}{\longmapsto} [x_1:\dots:x_{n+1}:x_{n}].
$$
So we have,
$
 i(\partial K)=\tilde P^{-1}(i(P(\partial K))),
$
which yields that
$
\ol{i(\partial K)}=\tilde P^{-1}(\ol{i(P(\partial K))})
$
because $\tilde P$ is a homeomorphism. Further note that $\ol{i(P(\partial K))}=i(\ol{P(\partial K)})$, because $P(\partial K)$ is bounded, and $i$ is a homeomorphism restricted to any compact subset of $\R^n$. So we conclude that 
$$
\ol{i(\partial K)}=\tilde P^{-1}\big(i(\ol{P(\partial K)})\big)
$$
which completes the proof.

 \section{Examples (and Counterexamples)}
 Here we gather a number of  examples which establish the optimality of various aspects of the theorems discussed in the introduction. All curves described below may be rotated around their axis of symmetry to obtain surfaces with analogous properties.

 \begin{example}\label{ex:tub}
It is easy to see that without the assumption on multiplicity in Theorem \ref{thm:main1}, the set $X$ may not be a hypersurface; see Figure \ref{fig:y}, specially the middle diagram, which also demonstrates that the value of $m$ in Theorem \ref{thm:main1} is optimal. Further, the collection of hypersurfaces which make up $X$ may not be finite, even locally.  Figure \ref{fig:tub} 
\begin{figure}[h]
 \centering
   \begin{overpic}[height=1in]{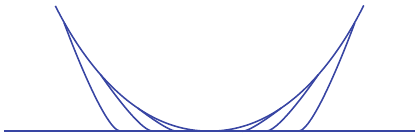} 
  \end{overpic}
  \caption{}
  \label{fig:tub}
  \end{figure}
shows one such example.  It has continuously varying flat tangent cones, but cannot be decomposed into a finite number of hypersurfaces near the neighborhood surrounding the apex 
of the parabola.
\end{example}

 \begin{example}\label{ex:deltoids}
For any given $\alpha>0$,  there 
is a convex real algebraic hypersurface which is not  $\C^{1,\alpha}$. Explicit examples are given by the convex  curves $y^{2n-1}=x^{2n}$, where $n=1$, $2$, $3$, \dots. These curves are $\C^1$ by Theorem \ref{thm:main1}, and are $\C^\infty$ everywhere except at the origin $o$. But they  are not $\C^{1,\alpha}$, for $\alpha>1/(2n-1)$, in any neighborhood of $o$.
All these properties are shared by the projectively equivalent family 
of closed convex curves $(1-y)y^{2 n - 1} = x^{2 n}$ which is depicted in Figure \ref{fig:deltoids}, for $n=2$, $3$, $4$. Here the singular point $o$ lies at the bottom of each curve.
 \begin{figure}[h]
   \centering
    \begin{overpic}[height=1in]{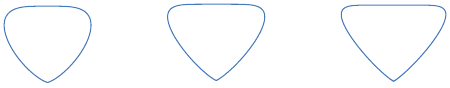} 
   \end{overpic}
   \caption{}
   \label{fig:deltoids}
\end{figure}
\end{example}

 \begin{example}\label{ex:analytic}
There are real algebraic hypersurfaces with flat  tangent cones which  are supported by  balls (of nonuniform radii) at each point but are not $\C^1$, such as the sextic surface $z^3=x^5 y+xy^5$
depicted in Figure \ref{fig:saddle2}.
This surface is regular in the complement of the origin. So it has  support balls there. Further,  one might directly check that there is even a support ball  at the origin. On the other hand the surface is not $\C^1$, since all its tangent planes  along the $x$ and $y$ axis are vertical, while at the origin the tangent plane is horizontal.
 \begin{figure}[h]
   \centering
   \subfloat[]{\label{fig:saddle2}\includegraphics[height=1.5in]{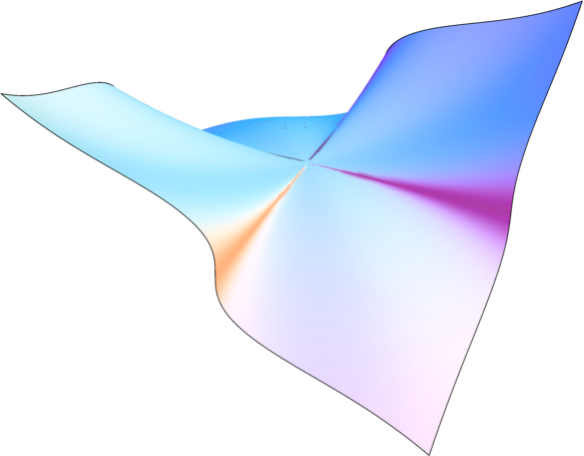}}\hspace{0.75in}
   \subfloat[]{\label{fig:saddle1}\includegraphics[height=1.4in]{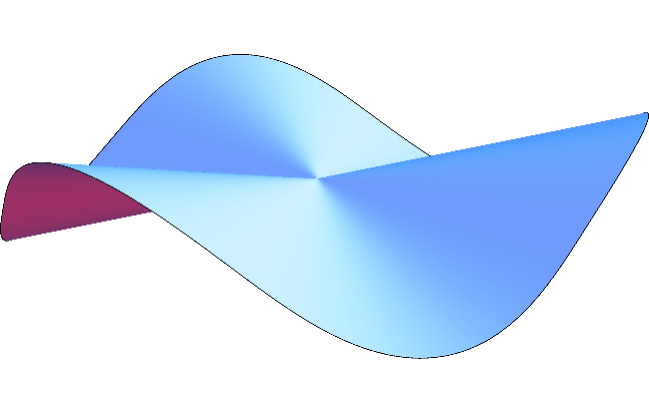}}
         \caption{}
\end{figure}
\end{example}

 \begin{example}\label{ex:fermat} 
There are real algebraic hypersurfaces whose tangent cones are hypersurfaces but are not hyperplanes, such as the Fermat cubic
$x^3+y^3=z^3,$
see Figure \ref{fig:saddle1}.
All points of this surface, except the origin, are regular and therefore the tangent cones are flat everywhere in the complement of the origin. On the other hand, the tangent cone at the origin is the surface itself, since the surface is invariant under homotheties. 
\end{example}

\begin{example}\label{ex:dingdong}
There are real semialgebraic convex hypersurfaces in $\R^n$ which are not $\C^1$. An example is the portion of the ``Ding-dong  curve" \cite{hauser} given by 
$x^2 = y^2 (1 - y)$ and $y\geq 0$; see Figure \ref{fig:dingdong1}. Also note that this curve is projectively equivalent to $y (1 - x^2) = 1$, $-1\leq x\leq 1$, depicted in Figure \ref{fig:dingdong2}, which shows that there are semialgebraic strictly convex hypersurfaces which are homeomorphic to $\R^{n-1}$ but are not entire graphs.
\begin{figure}[h]
     \centering
   \subfloat[]{\label{fig:dingdong1}\includegraphics[height=1.1in]{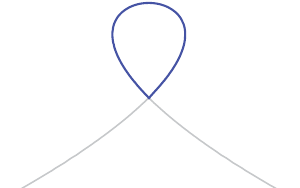}}\hspace{0.5in}
   \subfloat[]{\label{fig:dingdong2}\includegraphics[height=1.1in]{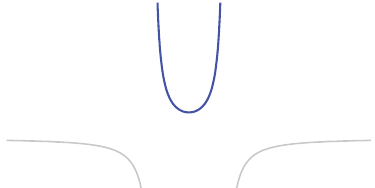}}
      \caption{}
\end{figure}
\end{example}

 \begin{example}\label{ex:horseshoe}
 There are real analytic convex hypersurfaces homeomorphic to $\R^{n-1}$ whose projective closure is not $\C^1$ or even differentiable; for instance,
$
x^2+e^{-y}=1
$
defines an unbounded convex  planar curve which  is contained within the slab $-1\leq x\leq 1$ and thus is not an entire graph, see Figure \ref{fig:horseshoe1}. 
\end{example}

 \begin{example}\label{ex:teardrop}
There are smooth real algebraic  hypersurfaces which are homeomorphic to $\R^{n-1}$ but are not entire graphs. Consider for instance the planar curve
$
y (1-x^2 y )=1
$
 pictured in Figure \ref{fig:horseshoe2}. This curve is projectively equivalent to  the ``pear-shaped quartic" or \emph{Piriform}   given by $x^2=y^3(1-y)$, see Figure \ref{fig:tear}.
\end{example}

\begin{example}\label{ex:pcl}
There are real algebraic convex hypersurfaces homeomorphic to $\R^{n-1}$ whose projective closure is not  $\C^1$ or even a topological hypersurface. Consider for instance the parabolic cylinder $P\subset\R^3$ given by
$
2z=x^2+1
$.
 This surface is projectively equivalent to the circular cylinder $C$ given by $(z-1)^2+x^2=1$, via the transformation $(x,y,z)\mapsto (x/z,y/z,1/z)$. Translating $C$, we obtain $C'$ given by $z^2+x^2=1$. Now consider the projective transformation $(x,y,z)\mapsto (x/y,1/y,z/y)$ which maps $C'$ to the cone $z^2+x^2=y^2$. Since, as we discussed in Section \ref{subsec:projtranform}, these projective transformations extend to diffeomorphism of $\RP^3$, we then conclude that that the projective closure of our original surface $P$ had a conical singularity.
\end{example}

\begin{figure}[h]
     \centering
   \subfloat[]{\label{fig:horseshoe1}\includegraphics[height=1.1in]{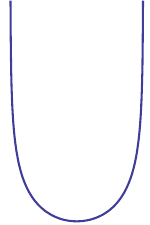}}\hspace{1in}
   \subfloat[]{\label{fig:horseshoe2}\includegraphics[height=1.1in]{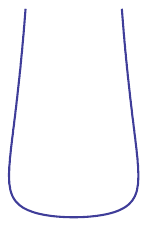}}\hspace{1in}
     \subfloat[]{\label{fig:tear}\includegraphics[height=1.1in]{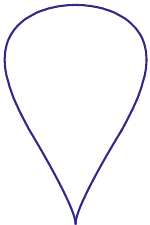}}
      \caption{}
\end{figure}

\appendix
\section{Another Proof for\\ The Regularity of Real Analytic Convex Hypersurfaces}\label{sec:appendix}
Using the classical resolution of singularities for planar curves, we describe here an  alternative proof of the $\C^1$-regularity of real analytic convex hypersurfaces (which was a special case of Theorem \ref{thm:main3}).  More specifically, we employ the following  basic fact, c.f.  \cite[Lemma 3.3]{milnor:singular}, which may be considered the simplest case of Hironaka's uniformization theorem \cite{BM,hironaka}. By a \emph{real analytic curve} here we mean a real analytic set of dimension one.

\begin{lem}[Newton-Puiseux \cite{griffiths,milnor:singular}]
Let $\Gamma\subset\R^2$ be a real analytic curve and $p\in\Gamma$ be a nonisolated point. Then there is an open neighborhood $U$ of $p$
in $\R^2$ such that $\Gamma\cap U=\cup_{i=1}^k \Gamma_i$ where each ``branch" $\Gamma_i$ is homeomorphic to $\R$ via a real analytic (injective) parametrization $\gamma_i\colon (-1,1)\to \Gamma_i$.\qed
\end{lem}
See \cite[p. 76--77]{griffiths} for the proof of the above lemma in the complex case, which in turn yields  the real version as described in \cite[p. 29]{milnor:singular}. Other treatments of the complex case may also be found in \cite{casas,wall}. The main ingredient here is Puiseux's decomposition theorem for the germs of  analytic functions, which goes back to Newton, see \cite[Thm. 1.1]{kollar} or \cite[Sec. 2.1]{cutkosky}.   
Next note that any simple planar curve which admits an analytic parametrization must be piecewise smooth, because the speed of the parametrization can vanish only on a discrete set; furthermore, the curve may not have any ``corners" at these singularities:
\begin{lem}
Let $\gamma\colon (-1,1)\to\R^2$ be a nonconstant real analytic map, and set $T(t):=\gamma'(t)/\|\gamma'(t)\|$.
Then 
$
\lim_{t\to a^+}T(t)=\pm \lim_{t\to a^-}T(t)
$
for all $a\in(-1,1)$. In particular, $\gamma$  has  continuously turning tangent lines.
\end{lem}

By ``tangent line" here we   mean the symmetric tangent cone of the image of $\gamma$ (Section \ref{subsec:prelim1}). Note also that  the analyticity assumption in the above lemma  is essential (the curve $y=|x|$, $-1<x<1$, for instance, admits the well-known $\C^\infty$ parametrization given by $\gamma(t):=e^{-1/t^2}(t/|t|,1)$ for $t\neq 0$, and $\gamma(0):=(0,0)$).
\begin{proof}
We may assume $a=0$.  If $\|\gamma'(0)\|\neq 0$, then the proof immediately follows. So suppose that $\|\gamma'(0)\|= 0$. Then, by analyticity of $\gamma'$, there is an integer $m> 0$ and an analytic 
map $\xi \colon(-\epsilon,\epsilon)\to \R^2$ with $\|\xi(0)\|\neq 0$ such that $\gamma'(t)=t^m\xi(t)$. Thus
$$
T(t)=\frac{t^m\xi(t)}{\|t^m\xi(t)\|}
=
\left(\dfrac{t}{|t|}\right)^m
\dfrac{\xi(t)}{\|\xi(t)\|},
$$
which in turn yields:
$$
\lim_{t\to 0^+}T(t)
=
1^m\dfrac{\xi(0)}{\|\xi(0)\|}=(-1)^{2m}\dfrac{\xi(0)}{\|\xi(0)\|}=(-1)^m \lim_{t\to 0^-}T(t).
$$
\end{proof}

The last two lemmas   yield the following basic fact which  generalizes Corollary \ref{cor:cusp}  obtained  earlier from Theorem \ref{thm:main3}. Recall that a simple planar curve has a \emph{cusp} at some point if its tangent cone there consists of a single ray (c.f. Section \ref{subsec:cusp}).

\begin{cor}\label{cor:appendix}
Each branch of a real analytic curve $\Gamma\subset\R^2$ at a nonisolated point $p\in\Gamma$ is either $\C^1$ near $p$ or  has a cusp at $p$.\qed
\end{cor}

This observation quickly shows that a convex real analytic curve $\Gamma\subset\R^2$ is $\C^1$, since by Lemma \ref{lem:convexsurface} it cannot have any cusps. This in turn yields the same result in higher dimensions, via a slicing argument, as we describe next. Let $X\subset\R^n$ be a real analytic convex hypersurface. There exists then a convex set $K\subset\R^n$ with $\partial K=X$. Thus, by Lemma \ref{lem:convexity}, to establish the regularity of $X$ it suffices to show that through each point $p$ of $X$ there passes a unique support hyperplane. Suppose, towards a contradiction, that there are two distinct support hyperplanes $H_1$ and $H_2$ passing through $p$. Then $L:=H_1\cap H_2$ has dimension $n-2$. Now let $o$ be a point in the interior of $K$. Since $o\not\in L$, there exists a (two dimensional) plane $\Pi\subset\R^n$ which 
is transversal to $L$ at $p$, and 
passes through $o$. Consequently $\Gamma:=X\cap \Pi=\partial (K\cap\Pi)$ is a convex real analytic planar curve, and therefore must be $\C^1$ (by Corollary \ref{cor:appendix}). So $\Gamma$ must have exactly one support line at $p$, which is a contradiction since $\ell_i:=H_i\cap\Pi$ are distinct support lines of $\Gamma$ at $p$.

\section*{Acknowledgments}
We thank Matt Baker, Saugata Basu, Igor Belegradek, Eduardo Casas Alvero, Joe Fu, Frank Morgan, Bernd Sturmfels, Serge Tabachnikov, and Brian White,  for useful communications. Thanks also to the anonymous referee for informing us about Alexander Lytchak's work, and its connection to the last claim in Theorem \ref{thm:main2}.

\bibliographystyle{abbrv}
\bibliography{references}

\end{document}